\documentclass[reqno,12pt]{article}

\usepackage{a4wide}
\usepackage{amsmath,stmaryrd} 
\usepackage{amssymb}
\usepackage{amsthm}
\usepackage[utf8]{inputenc} 
\usepackage{graphicx} 
\usepackage{xcolor}
\usepackage[normalem]{ulem}

\numberwithin{equation}{section}

\newtheorem{theo}{Theorem}

\newtheorem{coro}{Corollary}
\newtheorem{prop}{Proposition}

\newtheorem{lem}{Lemma}

\newtheorem{Question}{Question}

\theoremstyle{remark}

\newtheorem*{Remark*}{Remark}
\newtheorem*{Remarks*}{Remarks}

\makeatletter
\newcommand*{\house}[1]{%
 \mathord{%
 \mathpalette\@house{#1}%
 }%
}
\newcommand*{\@house}[2]{%
 \dimen@=\fontdimen8 %
 \ifx#1\scriptscriptstyle\scriptscriptfont
 \else\ifx#1\scriptstyle\scriptfont
 \else\textfont\fi\fi
 3 %
 \sbox0{%
 $#1%
 \vrule width\dimen@\relax
 \overline{%
 \kern2\dimen@
 \begingroup 
 #2%
 \endgroup
 \kern2\dimen@
 }%
 \vrule width\dimen@\relax
 \mathsurround=1.5\dimen@ 
 $%
 }%
 \ht0=\dimexpr\ht0-\dimen@\relax
 \dp0=\dimexpr\dp0+2\dimen@\relax
 \vbox{%
 \kern\dimen@ 
 \copy0 %
 }%
}

\newcommand{\tra}{{}^t}
\newcommand{\N}{\mathbb{N}}

\newcommand{\Q}{\mathbb{Q}}

\newcommand{\C}{\mathbb{C}}
\newcommand{\K}{\mathbb{K}}

\newcommand{\Qbar}{\overline{\mathbb Q}}

\newcommand{\etoile}{^*}

\newcommand{\calP}{\mathcal P}

\newcommand{\calD}{\mathcal D}

\newcommand{\EE}{{\bf E}}
\newcommand{\GG}{{\bf G}}

\newcommand{\Span}{{\rm Span}}

\begin{document}

\title{On the singularities of differential equations satisfied by $E$-functions}
\date\today
\author{S. Fischler and T. Rivoal}
\maketitle

\begin{abstract} 
Let $\xi$ be a value, at an algebraic point, of a Siegel $E$-function. As a special case of a very general interpolation result, we prove that there exists an $E$-function $f$ such that $f(1)=\xi$, and such that 1 is not a singularity of the minimal differential equation satisfied by $f$. We prove that the same property does not hold at the point $0$, when $\xi$ is the value at a non-zero algebraic number of the Bessel function. This answers an analogue of a question asked by Yves André for $G$-functions.
\end{abstract}

\section{Introduction}

 We first recall the definition of $E$-functions, due to Siegel~\cite{siegel}. We embed the field of algebraic numbers $\Qbar$ in $\mathbb C$. 
A power series $f(x)=\sum_{n=0}^{\infty} a_n x^n/n! \in \Qbar[[x]]$ is said to be an $E$-function~if 

$(i)$ $f(x)$ is solution of a non-zero linear differential equation with coefficients in $\Qbar(x)$.
\smallskip

$(ii)$ For all $\varepsilon>0$, there exists an integer $N(\varepsilon)$ such that for all $n\ge N(\varepsilon)$, all Galois conjugates of $a_n$ have modulus less than $n!^{\varepsilon}$. 

\smallskip

$(iii)$ There exists a sequence of positive integers $d_n$ such that $d_na_m$ are algebraic integers for all~$m\le n$ and such that for all $\varepsilon>0$, there exists an integer $N(\varepsilon)$ such that for all $n\ge N(\varepsilon)$, $d_n\le n!^{\varepsilon}$.

\medskip

 An $E$-function is either a polynomial or a transcendental function. 
In $(i)$, it does not matter if the differential equation is inhomogeneous or homogeneous, because from an inhomogeneous equation of order $\mu$ for $F$ we readily obtain a homogeneous one of order $\mu+1$. Unless otherwise stated, all the differential operators considered in this text  will be implicitly assumed to be non-zero and have   coprime  coefficients in $\Qbar[x]$. An {\em a priori} smaller class of $E$-functions, called {\em strict} $E$-functions, has also been considered in the literature: $n!^{\varepsilon}$ for all $\varepsilon>0$ is replaced with $c^n$ for some $c>1$ (see for instance \cite{andre}); 
 it is conjectured that both classes coincide. $E$-functions below will be understood in Siegel's sense. The statements of our results and their proofs hold {\em mutatis mutandis} for $E$-functions in the strict sense. 

\medskip

When $f(x) =\sum_{n=0}^{\infty} a_n x^n/n!  $ is an $E$-function in the strict sense, the function $\sum_{n=0}^{\infty} a_n x^n  $ is called a $G$-function. We are interested in the values taken at algebraic points by these functions. Let  $\EE := \{f(\alpha) : \, f \textup{ is an {\it E}-function and } \alpha\in\Qbar\}$, and define $\GG$ in the same way where $f$ is an analytic continuation of a $G$-function.  
These sets, which contain $\Qbar$, were defined and proved to be rings in \cite{gvalues,ateo} (when $E$-functions are considered in the strict sense;   we extend here the definition to $E$-functions in  Siegel's sense and it is still a ring).

\medskip

The ring $\GG$ of values of $G$-functions was studied more deeply in \cite{gvalues}. It is related to the ring $\calP$ of periods (in the sense of \cite{KZ}, say): conjecturally, $\GG = \calP[1/\pi]$ (see \cite[end of \S 2.2]{gvalues}). The following result was proved in \cite{gvalues}:

\begin{theo}\label{thG}
    For any $\xi\in\GG$ there exists a $G$-function $f(x) =\sum_{n=0}^{\infty} a_n x^n  $  with coefficients $a_n$ in $\Q(i)$ such that $f(1)=\xi$; moreover the radius of convergence of such a power series can be arbitrarily large. 
\end{theo}

The following question was asked by Y. André (for the point $0$), and is  open: 

\begin{Question} \label{q1} Given $\xi\in\GG$, does  there always exist a $G$-function $f(x)$ such that $f(1)=\xi$ and 
$f$ is solution of a homogeneous differential equation of which the point $0$ (resp. the point $1$) is not a singularity?
\end{Question}

\medskip

In the present paper, we study analogous statements for $E$-functions.  The ring $\EE$ of values of $E$-functions is related to, and conjecturally contained in,  the ring of  exponential periods (see \cite{KZ,FJ}). Any $E$-function is entire, so the assertion about the radius of convergence in Theorem~\ref{thG} is meaningless for $E$-functions. Restricting to coefficients $a_n$ in $\Q(i)$, or in a given number field $\K$, yields a ring of values strictly contained in $\EE$ (see \cite[Theorem 4]{ateo} and \cite{FRLiouville}) so this part of Theorem~\ref{thG} is completely different for $E$-functions. On the other hand,  the analogue of Question \ref{q1} makes sense for $E$-functions and we shall answer it.

\medskip

To begin with, we give a negative answer to the analogue of Question \ref{q1} at the point $0$.

\begin{theo}\label{thz} Denote by $J_0(x) := \sum_{n=0}^{\infty} \frac{(-1)^n}{n!^2} ( x/2 )^{2n}$ the Bessel function. Let $\alpha\in\Qbar\etoile$, and $f$ be any $E$-function such that $f(1)=J_0(\alpha)$. Then $0$ is a non-apparent singularity of any non-zero (in)homogeneous differential equation satisfied by $f$. 
\end{theo}

In this theorem, and throughout the paper, in writing {\em (in)homogeneous} we mean that the statement holds with both {\em homogeneous} and {\em inhomogeneous} instead.

Recall that an inhomogeneous differential equation is an equation of the form $L_0y=P_0$ where
 $L_0  = \sum_{i=0}^{\mu_0} a_i(x) (\frac{d}{dx})^i \in \Qbar[x, \frac{d}{dx}]\setminus\{0\}$ and $P_0\in \Qbar[x]$; it is homogeneous if $P_0=0$. We may 
  assume that $a_{\mu_0}\neq 0 $ and $\gcd(a_0,\ldots,a_{\mu_0},P_0)=1$; then a singularity of the equation $L_0y=P_0$ is, by definition, a zero of $a_{\mu_0}$.

We recall also that a singularity is {\em apparent} if the corresponding differential equation has a basis of solutions holomorphic at that point.   

\medskip

At the point $1$, a positive answer to the analogue of Question \ref{q1} follows from André's results \cite{andre} on $E$-operators (generalized to $E$-functions in Siegel's sense in \cite{lepetit}). Indeed, such an operator has no non-zero finite singularity, and any $E$-function is a solution of an $E$-operator. However, the minimal homogeneous differential equation of an $E$-function may have a singularity at $1$ (which is then apparent since it is a right-factor of an $E$-operator); this happens for instance with $f(x) = (x-1)e^x$. We may strengthen Question \ref{q1} by asking the  minimal (in)homogeneous differential equation of $f$ not to have a singularity at the point 1. We shall prove the the following result, which gives a positive answer to this stronger question:

\begin{theo}\label{thu} For any $\xi\in \EE$, there is an $E$-function $f$ such that $f(1)=\xi$ and 1 is not a singularity of the minimal (in)homogeneous differential equation satisfied by $f$.
\end{theo}

This result holds trivially if $\xi$ is algebraic. Otherwise it is a special case (with $N=1$, $\alpha_1=1$, $T=1$, $\xi_{1,0}=\xi$) of the following much more general interpolation theorem.

\begin{theo} \label{thinterp}
Let $N,T\geq 1$, and $\alpha_1,\ldots,\alpha_N$ be pairwise distinct non-zero algebraic numbers. Let $\xi_{n,t}\in\EE$ for $1\leq n \leq N$ and $0\leq t \leq T-1$. Then there exists an $E$-function $f$ such that:
\begin{itemize}
\item For any $1\leq n \leq N$ and  any $0\leq t \leq T-1$, we have $f^{(t)}(\alpha_n) = \xi_{n,t}$.
\item Let $Ly=0$, with $L\in \Qbar[x,\frac{d}{dx}]$, be the minimal homogeneous differential equation satisfied by $f$; denote by $\mu$ its order. Then $\mu \geq T+1$ and for any $n$, we have the following equivalence: $\alpha_n$ is a singularity of  $L$ if, and only if, the values $\xi_{n,t}$ (for $0\leq t \leq T-1$) are linearly dependent over $\Qbar$. 
\item Let $L_0y=P_0$, with $L_0\in \Qbar[x,\frac{d}{dx}]$ and $P_0\in\Qbar[x]$, be the minimal inhomogeneous differential equation satisfied by $f$. Then $L_0$ has order $\mu-1$ and for any $n$, we have the following equivalence: $\alpha_n$ is a singularity of  $L_0$ if, and only if, the values 1 and $\xi_{n,t}$ (for $0\leq t \leq T-1$) are linearly dependent over $\Qbar$.
\end{itemize}
\end{theo}

\begin{Remark*}  Let $\K$ be a number field that contains $\alpha_1,\ldots,\alpha_N$, and such that for any $n$ and any $t$ there exists an $E$-function $f_{n,t}$ with coefficients in $\K$ such that $f_{n,t}(1) =\xi_{n,t} $. Then the  $E$-function $f$ we construct has coefficients in $\K$. 
\end{Remark*}

\bigskip

The structure of this paper is as follows. In \S \ref{sec2} we prove Theorem~\ref{thz}, and then in \S \ref{sec3} we study non-zero singularities of minimal equations of $E$-functions, starting in \S \ref{subsec31} with general results of independent interest and proving  Theorem~\ref{thinterp} in \S \ref{subsec33}.

\section{Proof of Theorem~\ref{thz}} \label{sec2}

The proof of Theorem~\ref{thz} is based of the following result of André \cite{andre} for $E$-functions in the strict sense with rational coefficients, generalized by Beukers \cite{beukers} to all $E$-functions in the strict sense and by Lepetit \cite[p. 143, Proposition 18]{lepetit} to all $E$-functions in Siegel's sense.

\begin{theo} \label{thAB} Let $f$ be an $E$-function such that $f(1)=0$. Then $f(x)/(x-1)$ is an $E$-function.
\end{theo}

We shall use also the fact that the Bessel function $J_0$ is a solution of  $Ly=0$, where
\begin{equation}
    \label{eq0thd} 
L := x \Big( \frac{d}{dx}\Big)^2 +    \frac{d}{dx}  + x.
\end{equation}
Of course $0$ is a singularity of $L$; a basis of local solutions at $0$ is given by $(J_0(x),  J_0(x)\log x + \tilde{J}_0(x))$ where   $\tilde{J}_0$ is an $E$-function.

Let $f$ be an  $E$-function such that $f(1)=J_0(\alpha)$. Theorem~\ref{thAB} applied to the function $f(x) - J_0(\alpha x)$ provides an $E$-function $g$ such that 
\begin{equation}
    \label{eq1thd} f(x) = J_0(\alpha x) + (x-1)g(x).
\end{equation}
Let $\calD\in\Qbar[x, \frac{d}{dx}]$ be a non-zero differential operator that annihilates  $f(x)$, $J_0(\alpha x) $ and $ g(x)$. We denote by $V$ a Picard-Vessiot extension of $\Qbar(x)$ that contains a basis of solutions of $\calD$, and by $G$ the corresponding differential Galois group. For simplicity we write $J(x) = J_0(\alpha x)$.

Let $\sigma\in G$. Then $\sigma$ can be seen as a field automorphism of $V$ that commutes with derivation and leaves any element of $\Qbar(x)$ invariant. Applying $\sigma$ to Eq.~\eqref{eq1thd} yields
\begin{equation}
    \label{eq2thd} (\sigma f)(x) = (\sigma J)(x) + (x-1)(\sigma g)(x).
\end{equation}
The function $\sigma J$ is also a solution of the differential operator obtained from \eqref{eq0thd} by changing $x$ to $\alpha x$, so there exist scalars $\delta,\beta$ such that $\sigma J=   \delta J+ \beta (J \log   + \tilde{J})$ where $\tilde{J}$ is an $E$-function. On the other hand, there exists an $E$-operator $\calD_g$ of which $g$ is a solution, and $\sigma g$ is also a solution of $\calD_g$ so that $\sigma g$ is a Nilsson-Gevrey arithmetic series of order $-1$ (see \cite{andre}). Accordingly, it can be written as  a finite sum
$$\sum_{\nu,j} \lambda_{\nu,j} h_{\nu,j}(x) x^\nu (\log x)^j $$
where $\nu\in\Q$, $j\in\N$, $\lambda_{\nu,j}\in\C$, and $ h_{\nu,j}$ is an $E$-function. At last, assume that $f$ is solution of an (in)homogeneous differential operator of which $0$ is not a singularity (or is an apparent singularity). Then $\sigma f $ is a solution of the same operator, so that  $\sigma f $ is holomorphic at $0$.

Expanding both sides of Eq.~\eqref{eq2thd} as polynomials in $\log x$, and taking the coefficient of $\log x$, yields
\begin{equation}
    \label{eq3thd} 0 =  \beta J (x) + (x-1) \sum_{\nu} \lambda_{\nu,1} h_{\nu,1}(x) x^\nu.
\end{equation}
Now all $E$-functions $h_{\nu,1}$ are holomorphic at $1$, so that taking $x=1$ in Eq.~\eqref{eq3thd} yields $\beta=0$ (since $J(1)= J_0(\alpha)\neq 0$ because it is a transcendental number for all $\alpha\in\Qbar\etoile$ by Siegel's theorem \cite{siegel}). 

\medskip

We have proved that for any $\sigma\in G$ there exists a scalar $\delta$ such that $\sigma J = \delta J$. This implies
$$
\sigma\Big( \frac{  J'}{  J}\Big) = \frac{ \sigma(J)'}{\sigma(J)} = \frac{\delta J'}{\delta J}= \frac{  J'}{  J}
$$
so that $J'/J\in\Qbar(x)$ (see for instance \cite[Proposition 1.9 and Theorem 1.6]{DBJAW}). This contradicts another result of Siegel \cite{siegel} that $J_0$ and $J_0'$ are algebraically independent over  $\Qbar(x)$.

\section{Non-zero singularities: proof of Theorem~\ref{thinterp}}\label{sec3}

\subsection{Non-zero singularities and linear dependence of values}\label{subsec31}

The following version of the Siegel-Shidlovskii theorem is due to Beukers \cite[Corollary 1.4]{beukers} in the strict sense, and to Andr\'e \cite{andre2} in Siegel's sense.

\begin{theo} \label{thbeukers}
Let  $f_{1}$, \ldots, $f_M$ be $E$-functions, linearly independent over $\Qbar(x)$, such that the  vector $\tra (f_{1} , \ldots,  f_M)$ is solution of a first-order linear differential system with entries in $\Qbar(x)$. 

Let $\alpha\in \Qbar\etoile$. If $\alpha$ is not a singularity of this system, then  $f_1(\alpha)$, \ldots, $f_M(\alpha)$ are  linearly independent over $\Qbar$.
\end{theo}
(By {\em singularity}, we mean a pole of one of the entries of the matrix of the system.) In particular, let $f$ be an $E$-function. Denote by $\mu$ the minimal order of a non-zero homogeneous linear differential equation satisfied by $f$. Then Theorem~\ref{thbeukers} applies to the functions $f_j := f^{(j-1)}$ with $M=\mu-1$, so that:
\begin{center}
 \em{If $\alpha$ is not a singularity of the non-zero minimal linear differential equation of $f$, then\\
 $f(\alpha)$, $f'(\alpha)$,  \ldots, $f^{(\mu-1)}(\alpha)$ are linearly independent over $\Qbar$.}
\end{center}

The following observation is a converse statement; it is not difficult to prove but we did not find explicitly  it in the literature. It is very convenient to ensure that a point is not a singularity of the minimal differential equation of an $E$-function; it will be used in the proof of Theorem~\ref{thinterp}.

\begin{prop}
\label{propsing}  
Let $L  = \sum_{i=0}^\mu a_i(x) (\frac{d}{dx})^i \in \Qbar[x, \frac{d}{dx}]\setminus\{0\}$ be a  differential operator such that $\gcd(a_0,\ldots,a_\mu)=1$. 

Let $\alpha\in \Qbar$ and $f$ be a solution of the differential equation $Ly=0$, defined around $\alpha$, such that $f(\alpha)$, $f'(\alpha)$,  \ldots, $f^{(\mu-1)}(\alpha)$ are linearly independent over $\Qbar$. Then $a_\mu(\alpha)\neq 0$ so that $\alpha$ is not a singularity of this differential equation $Ly=0$.
\end{prop}

\begin{proof}
On the contrary, assume that  $a_\mu(\alpha)=0$; then we have $\sum_{i=0}^{\mu -1} a_i(\alpha) f^{(i)}(\alpha)=0$. This is a $\Qbar$-linear relation between $f(\alpha)$, $f'(\alpha)$,  \ldots, $f^{(\mu-1)}(\alpha)$. Since by assumption $f(\alpha)$, $f'(\alpha)$,  \ldots, $f^{(\mu-1)}(\alpha)$ are $\Qbar$-linearly independent, this relation has to be trivial, so that $a_i(\alpha) =0$ for any $i$. This contradicts the assumption that  $\gcd(a_0,\ldots,a_\mu)=1$. 
\end{proof}

Combining Theorem~\ref{thbeukers} and Proposition \ref{propsing} we obtain the following corollary.

\begin{coro} \label{corohom}
    Let $f$ be an E-function, and $Ly=0$ denote its minimal homogeneous linear differential equation. Let $\alpha\in\Qbar\etoile $. Then $\alpha$ is a singularity of $L$ if, and only if, $f(\alpha)$, $f'(\alpha)$,  \ldots, $f^{(\mu-1)}(\alpha)$ are linearly dependent over $\Qbar$, where $\mu$ is the order of $L$.
\end{coro}

We can also adapt this proof to the inhomogeneous setting.

\begin{prop} \label{propinhom}
Let $f$ be an E-function, and $L_0y=P_0$ denote its minimal non-zero inhomogeneous linear differential equation. Let $\alpha\in\Qbar\etoile $. Then $\alpha$ is a singularity of this equation $L_0y=P_0$ if, and only if, 1, $f(\alpha)$, $f'(\alpha)$,  \ldots, $f^{(\mu_0-1)}(\alpha)$ are linearly dependent over $\Qbar$, where $\mu_0$ is the order of $L_0$.
\end{prop}

\begin{proof} We write $L_0  = \sum_{i=0}^{\mu_0} a_i(x) (\frac{d}{dx})^i \in \Qbar[x, \frac{d}{dx}]$ with $a_{\mu_0}  \neq 0$, and we may assume that $\gcd(a_0,\ldots,a_{\mu_0},P_0)=1$. Then a singularity of the equation $L_0y=P_0$ is, by definition, a zero of $a_{\mu_0}$.

$\Rightarrow$ We have $\sum_{i=0}^{\mu_0-1} a_i(\alpha) f^{(i)}(\alpha) =P_0(\alpha)$ since $a_{\mu_0}(\alpha)=0$. This is a $\Qbar$-linear relation between 1 and $f(\alpha)$, $f'(\alpha)$,  \ldots, $f^{(\mu_0-1)}(\alpha)$, and it is non-trivial because $a_0$, \ldots, $a_{\mu_0}$, $P_0$ have no common zero.

$\Leftarrow$ If $\alpha$ is not a singularity, we apply Theorem~\ref{thbeukers} to the functions $1$, $f$, $f'$, \ldots, $f^{(\mu_0-1)}$: these functions are linearly independent over $\Qbar(x)$ by minimality of the inhomogeneous  differential equation $L_0y=P_0$. We deduce that the values at $\alpha$ of these functions are linearly independent over $\Qbar$.
\end{proof}

\subsection{Beukers' desingularization lemma}

We shall use the following desingularization lemma. It is due to Beukers \cite[Theorem 1.5]{beukers} for $E$-functions in the strict sense and  was extended by Lepetit \cite[\S6.5, Th\'eor\`eme 13]{lepetit} to  $E$-functions in Siegel's sense.

\begin{lem} \label{propdesing}
Let  $g_{1}$, \ldots, $g_M$ be $E$-functions, linearly independent over $\Qbar(x)$, such that the  vector $\tra (g_{1} , \ldots,  g_M)$ is solution of a first-order linear differential system with entries in $\Qbar(x)$. Then there exist $E$-functions $h_{1}$, \ldots, $h_M$ such that:
\begin{itemize}
\item The vector $\tra (h_{1} , \ldots,  h_M)$ is solution of a first-order linear differential system with coefficients in $\Qbar[x,1/x]$ (that is, without non-zero finite singularity).
\item Each function $g_\ell$ can be written as a linear combination of $h_{1} , \ldots,  h_M$ with coefficients in  $\Qbar[x]$.
\end{itemize}
\end{lem}

\subsection{Completion of the proof of Theorem~\ref{thinterp}}\label{subsec33}

To begin with, we notice that if we prove Theorem~\ref{thinterp} with an additional point  $\alpha_{0}\in\Qbar\etoile$, then we can eventually forget about this point and the conclusion of Theorem~\ref{thinterp} will hold with the points 
$\alpha_{n}$, $1 \leq n \leq N$. To this additional 
point (distinct from $\alpha_1$, \ldots, $\alpha_N$) we attach the values $\xi_{0,t}=e^{t+1}$ for 
$0\leq t \leq T-1$. 

\bigskip

For each $n$ and each $t$, we denote by $f_{n,t}$ an $E$-function  such that $f_{n,t}(\alpha_t) =\xi_{n,t} $; it exists since $h(ax)$ is an $E$-function whenever $h$ is an $E$-function and $a\in\Qbar$.

We denote by $A$ the $\Qbar[x]$-module generated by the constant function 1 and the $f_{n,t}$, {\em i.e.} the set of elements of the form 
$$
P + \sum_{n=0}^N\sum_{t=0}^{T-1} P_{n,t} f_{n,t},
$$
where $P$ and the $P_{n,t}$ are polynomials with algebraic coefficients. Then $A$ is a finitely generated torsion-free $\Qbar[x]$-module, and $\Qbar[x]$ is a principal ring, so that $A$ is free: there exist $\Qbar[x]$-linearly independent elements $g_1,\ldots,g_m\in A$ such that each $f_{n,t}$ can be written as 
 $$
 f_{n,t}(x)  = 
\sum_{\ell =1}^m Q_{n,t,\ell}(x) g_\ell(x)$$
with polynomials $Q_{n,t,\ell}\in \Qbar[x]$, and the function 1 can be written in this way too. Each $g_\ell$ is an $E$-function; in particular it is solution of a homogeneous linear differential equation of order $\mu_\ell$, say. Denote by $B$ the $\Qbar(x)$-vector space generated by the functions $g_\ell^{(j)}$ with $1\leq \ell \leq m$ and $0\leq j \leq \mu_\ell-1$; then $B$ is invariant under derivation. Since  $g_1,\ldots,g_m$ are $\Qbar(x)$-linearly independent elements, there exist $M\geq m$ and $E$-functions $g_{m+1}$, \ldots, $g_M$ such that $(g_1,\ldots,g_M)$ is a basis of $B$. In particular, 
\begin{itemize}
\item The $E$-functions $g_{1}$, \ldots, $g_M$ are linearly independent over $\Qbar(x)$.
\item The vector $\tra (g_{1} , \ldots,  g_M)$ is solution of a first-order linear differential system with entries in $\Qbar(x)$.
\end{itemize}

\bigskip

Applying Lemma \ref{propdesing} to $g_{1}$, \ldots, $g_M$,  we obtain   $E$-functions $h_{1}$, \ldots, $h_M$  such that the vector $\tra (h_{1} , \ldots,  h_M)$ is solution of a first-order linear differential system with coefficients in $\Qbar[x,1/x]$.  They are  $\Qbar(x)$-linearly independent since the vector space they span over $\Qbar(x)$ has dimension $M$ (it contains the  linearly independent  $E$-functions $g_{1}$, \ldots, $g_M$). Moreover each $f_{n,t}$ (and also the constant function 1) is a linear combination of $h_{1}$, \ldots, $h_M$   with coefficients in $\Qbar[x]$. It is very important here to have polynomial coefficients (and not only coefficients in $\Qbar(x)$) because it enables us to evaluate at $x=\alpha_n$ and obtain for any $0\leq n \leq N$ and  any $0\leq t \leq T-1$:
$$
\xi_{n,t}\in V_n, \mbox{ where } V_n := \Span_{\Qbar}( h_1(\alpha_n),\ldots,h_M(\alpha_n)).
$$
We have also $1\in V_n$.

\bigskip

Let $0\leq n \leq N$. Since $h_{1}$, \ldots, $h_M$   are   $\Qbar(x)$-linearly independent  $E$-functions that satisfy a first-order linear differential system of which $\alpha_n$ is not a singularity, the Andr\'e-Beukers theorem (Theorem~\ref{thbeukers} above) shows that $h_1(\alpha_n),\ldots,h_M(\alpha_n)$ are linearly independent over $\Qbar$. In other words, $\dim_{\Qbar} V_n = M$.

We recall from the beginning of the proof that  $\xi_{0,i} = e^{i+1}$ for any $0\leq i \leq T-1$. Therefore $1$,  $\xi_{0,0}$,  $\xi_{0,1}$, \ldots,  $\xi_{0,T-1}$   are $\Qbar$-linearly independent elements of $V_0$, so that $T+1\leq \dim_{\Qbar} V_0=M$.
We  claim that for any $0\leq n \leq N$  it is possible to choose values $\xi_{n,t}\in V_n$, for $T\leq t \leq M-1$, with the following properties:
\begin{itemize}
    \item if $\xi_{n,0}$, \ldots,   $\xi_{n,T-1}$ are linearly independent over $\Qbar$, then so are $\xi_{n,0}$, \ldots,   $\xi_{n,M-1}$ (and then this is a basis of $V_n$); 
    \item if 1, $\xi_{n,0}$, \ldots,   $\xi_{n,T-1}$ are linearly independent over $\Qbar$, then so are 1, $\xi_{n,0}$, \ldots,   $\xi_{n,M-2}$ (and then this is a basis of $V_n$ too). 
\end{itemize}

If 1, $\xi_{n,0}$, \ldots,   $\xi_{n,T-1}$ are linearly independent over $\Qbar$, we use the Steinitz exchange lemma  to construct $\xi_{n,t}\in V_n$, for $T\leq t \leq M-2$, in such a way that  1, $\xi_{n,0}$, \ldots,   $\xi_{n,M-2}$ is a basis of $V_n$; then we let  $\xi_{n,M-1}=1$.

Otherwise, if $\xi_{n,0}$, \ldots,   $\xi_{n,T-1}$ are linearly independent over $\Qbar$ and $1$ belongs to the subspace spanned by these numbers, we  construct $\xi_{n,t}\in V_n$, for $T\leq t \leq M-1$, in such a way that    $\xi_{n,0}$, \ldots,   $\xi_{n,M-1}$ is a basis of $V_n$.

At last, if  $\xi_{n,0}, \ldots,\xi_{n,T-1}$  are linearly dependent, we simply choose $\xi_{n,t}\in V_n$ arbitrarily for $T\leq t \leq M-1$. This concludes the proof of our claim.

\bigskip

Now we shall construct our $E$-function $f$ as 
\begin{equation} \label{eqpreuveinterpun}
 f(x)  = \sum_{i=1}^M S_i (x) h_i(x) \mbox{ with } S_1,\ldots,S_M\in\Qbar[x].
\end{equation}
To be able to choose the polynomials $S_i$ in a suitable way, we compute the $t$-th derivative  of $f$ for any $t\geq 0$:
\begin{equation} \label{eqpreuveinterpde}
 f^{(t)}(x)  = \sum_{i=1}^M S_i^{(t)}(x) h_i(x)  +  \sum_{i=1}^M \sum_{k=0}^{t-1}\binom  t k S_i^{(k)}(x)h_i^{(t-k)}(x).
\end{equation}

\bigskip

We shall construct now algebraic numbers $\xi_{n,t,i}$ such that 
 if  $S_i^{(t)}(\alpha_n)=\xi_{n,t,i}$ 
 for any $0\leq n \leq N$, $1\leq i \leq M$ and  $0\leq t \leq M-1$, then  $f^{(t)}(\alpha_n) = \xi_{n,t}$ for any $n$, $t$. 
 
For $t=0$, since $\xi_{n,0}\in V_n$ for any $n$, we may write $\xi_{n,0} = \sum_{i=1}^M \xi_{n,0,i} h_i(\alpha_n)$   for some $\xi_{n,0,i} \in \Qbar$. Then the desired 
  property $f(\alpha_n) = \xi_{n,0}$ follows at once from Eq.~\eqref{eqpreuveinterpun}  provided that $S_i(\alpha_n)=\xi_{n,0,i}$.

If the values   $S_i^{(t')}(\alpha_n)\in\Qbar$ have been chosen already for any $0\leq t' \leq t-1$, with $t\geq 1$, we write for any $n$:
\begin{equation} \label{eqpreuveinterptr}
\xi_{n,t} -  \sum_{i=1}^M \sum_{k=0}^{t-1}\binom  t k S_i^{(k)}(\alpha_n)h_i^{(t-k)}(\alpha_n) = \sum_{i=1}^M \xi_{n,t,i} h_i(\alpha_n) \mbox{  for some } \xi_{n,t,i} \in \Qbar.
\end{equation}
Indeed  $h_i^{(t-k)}$ is a linear combination of $h_1$, \ldots, $h_M$ with coefficients in $\Qbar[x,1/x]$ due to the differential system satisfied by these functions, so that $ h_i^{(t-k)}(\alpha_n)\in V_n$. Moreover $\binom  t k S_i^{(k)}(\alpha_n) \in\Qbar$  and $V_n$ is a $\Qbar$-vector space that contains $\xi_{n,t}$, so the left hand side of Eq.~\eqref{eqpreuveinterptr} belongs to $V_n$ and the algebraic coefficients  $\xi_{n,t,i}$  exist. Using Eq.~\eqref{eqpreuveinterpde} we obtain immediately that $S_i^{(t)}(\alpha_n) = \xi_{n,t,i} $ for all $i$ implies $f^{(t)}(\alpha_n) = \xi_{n,t}$.

 This concludes the proof of our claim; of course it is possible to find polynomials $S_1$, \ldots, $S_M$ such that $S_i^{(t)}(\alpha_n)$ is equal to that algebraic number $\xi_{n,t,i}$ for any $0\leq n \leq N$, $1\leq i \leq M$ and  $0\leq t \leq M-1$.
Then the function $f$ defined by Eq.~\eqref{eqpreuveinterpun} satisfies $f^{(t)}(\alpha_n) = \xi_{n,t}$ for any $n$, $t$. This proves the first assertion of Theorem~\ref{thinterp}.

\bigskip

Now recall that the values $f^{(t)}(\alpha_0) = \xi_{0,t}$, $0\leq t \leq M-1$, are linearly independent over $\Qbar$ since $e$, $e^2$, \ldots, $e^{T}$ are. Therefore the functions $f$, $f'$, \ldots, $f^{(M-1)}$ are linearly independent over $\Qbar [x]$ (that is, over $\Qbar(x)$). Denoting by $\mu$ the order of the minimal homogeneous differential equation satisfied by $f$, this implies $\mu \geq M$. On the other hand, the $\Qbar(x)$-vector space spanned by $h_1$, \ldots, $h_M$ has dimension $M$, is stable under taking  derivatives and contains $f$ so that $\mu \leq M$. Finally we have $\mu=M$, and we have noticed that $M\geq T+1$ so that $\mu\geq T+1$. 

Corollary \ref{corohom} then shows that for any $n$ we have the following equivalence: $\alpha_n$ is a singularity of the minimal homogeneous differential equation satisfied by $f$ if, and only if, the values $f^{(t)}(\alpha_n)=\xi_{n,t}$ (for $0\leq t \leq M-1$) are linearly dependent over $\Qbar$. By construction of the $\xi_{n,t}$ for $T\leq t \leq M-1$, this is equivalent to the linear dependence of $\xi_{n,t}$  for $0\leq t \leq T-1$.  Therefore  the second assertion of Theorem~\ref{thinterp} holds.

\bigskip

To conclude, let us consider the minimal inhomogeneous differential equation $L_0y=P_0$ satisfied by $f$, and denote by $\mu_0$ its order. The family $(f,f',\ldots, f^{(M-1)})$ is a basis of the $\Qbar(x)$-vector space spanned by $h_1$, \ldots, $h_M$, and this vector space contains the constant function $1$. This provides a non-trivial linear relation between $1$, $f$, $f'$, \ldots, $f^{(M-1)}$, that is an inhomogeneous differential equation of order at most $M-1$ satisfied by $f$. Denoting by $\mu_0$ the order of the minimal inhomogeneous differential equation $L_0y=P_0$ satisfied by $f$, we have $\mu_0\leq M-1$. Actually equality holds, because $f$ then satisfies an  homogeneous differential equation of order $\mu_0+1$ so that $\mu_0+1\geq M$.

Proposition \ref{propinhom} shows that for any $n$, $\alpha_n$ is a singularity of this equation $L_0y=P_0$ if, and only if, the values 1 and $f^{(t)}(\alpha_n)=\xi_{n,t}$ (for $0\leq t \leq M-2$) are linearly dependent over $\Qbar$. By construction of the $\xi_{n,t}$ for $T\leq t \leq M-1$, this is equivalent to the linear dependence of 1 and the $\xi_{n,t}$ for $0\leq t \leq T-1$. 

This concludes the proof of Theorem~\ref{thinterp}.

\noindent St\'ephane Fischler, Universit\'e Paris-Saclay, CNRS, Laboratoire de math\'ematiques d'Orsay, 91405 Orsay, France.

\medskip

\noindent Tanguy Rivoal,  Institut Fourier, Universit\'e Grenoble Alpes, CNRS,CS 40700, 38058 Grenoble cedex 9, France.

\bigskip

\noindent Keywords: $E$-functions,  Singularities of differential equations, Andr\'e-Beukers lifting theorem.

\bigskip

\noindent MSC 2020: 11J91 (Primary), 34M03 (Secondary)

\end{document}